\documentclass[11pt,amscd]{amsart}

\usepackage{amssymb}
\usepackage[english]{babel}

\newtheorem{theorem}{Theorem}[section]
\newtheorem{prop}[theorem]{Proposition}
\newtheorem{lemma}[theorem]{Lemma} 
\newtheorem{cor}[theorem]{Corollary}

\theoremstyle{definition}
\newtheorem{ex}[theorem]{Example}

\theoremstyle{remark} 
\newtheorem*{remark*}{Remark}

\newcommand{\tref}[1]{Theorem~\ref{#1}}
\newcommand{\pref}[1]{Proposition~\ref{#1}}
\newcommand{\lref}[1]{Lemma~\ref{#1}}
\newcommand{\cref}[1]{Corollary~\ref{#1}}

\newcommand{\eref}[1]{Equation~\eqref{#1}}
\newcommand{\exref}[1]{Example~\ref{#1}}

\renewcommand{\theenumi}{\alph{enumi}}

\numberwithin{equation}{section}
    
\newcommand{\NN}{{\mathbb N}} 
\newcommand{\ZZ}{{\mathbb Z}}
\newcommand{\QQ}{{\mathbb Q}}
\newcommand{\RR}{{\mathbb R}}
\newcommand{\w}{{\mathbf w}}
\newcommand{\m}{{\mathfrak m}}

\newcommand{\upto}{,\ldots,}
\newcommand{\Aut}{\operatorname{Aut}}
\newcommand{\LT}{\operatorname{LT}}
\newcommand{\LF}{\operatorname{LF}}
\newcommand{\lex}{\operatorname{lex}}
\newcommand{\height}{\operatorname{ht}}
\newcommand{\Spec}{\operatorname{Spec}}
\newcommand{\ini}{\operatorname{in}}
\newcommand{\bigheight}{\operatorname{bight}}

\begin{document}

\title{Krull dimension and Monomial Orders}
\author{Gregor Kemper}
\address{Technische Universi\"at M\"unchen, Zentrum Mathematik - M11,
Boltzmannstr. 3, 85748 Garching, Germany}
\email{kemper@ma.tum.de}

\author{Ngo Viet Trung}
\address{Institute of Mathematics, Vietnam Academy of Science and Technology, 18 Hoang Quoc Viet, 10307 Hanoi, Vietnam}
\email{nvtrung@math.ac.vn}


%\date{}

\begin{abstract}
We introduce the notion of independent sequences with respect to a monomial order 
by using the least terms of polynomials vanishing at the sequence. 
Our main result shows that the Krull dimension of a Noetherian ring is equal to the supremum of the length of independent sequences.	
%This result can be used to compute the Krull dimension effectively.
The proof has led to other notions of independent sequences, which have interesting applications. 
For example, we can show that $\dim R/0:J^\infty$ is the maximum number of analytically independent elements in an arbitrary ideal $J$ of a local ring $R$ and that $\dim B \le \dim A$ if $B \subset A$ are (not necessarily finitely generated) subalgebras of a finitely generated algebra over a Noetherian Jacobson ring. 
\end{abstract}

\maketitle

\section*{Introduction}

Let $R$ be an arbitrary Noetherian ring, where a ring is always assumed to be commutative with identity.
The aim of this paper is to characterize the Krull dimension $\dim R$
by means of a monomial order on polynomial rings over $R$. We are
inspired of a result of Lombardi  in \cite{Lombardi:2002} (see also
Coquand and Lombardi \cite{Coquand-Lombardi:03},
\cite{Coquand-Lombardi:05}) which says that for a positive integer
$s$, $\dim R < s$ if and only if for every sequence of elements $a_1
\upto a_s$ in $R$, there exist nonnegative integers $m_1 \upto m_s$ and
elements $c_1 \upto c_s \in R$ such that 
$$a_1^{m_1} \cdots a_s^{m_s}+c_1a_1^{m_1 + 1}  + c_2a_1^{m_1}a_2^{m_2+1} +  \cdots + c_sa_1^{m_1} \cdots a_{s-1}^{m_{s-1}} a_s^{m_s +1} =  0.$$
This result has helped to develop a constructive theory for the Krull dimension  \cite{Coquand-Lombardi-Roy}, \cite{Coquand-Ducos-Lombardi-Quitte}, \cite{Ducos:09}.\par

The above relation means that $a_1 \upto a_s$ is a solution of the polynomial
$$x_1^{m_1} \cdots x_s^{m_s}+c_1x_1^{m_1 + 1}  + c_2x_1^{m_1}x_2^{m_2+1} +  \cdots + c_sx_1^{m_1} \cdots x_{s-1}^{m_{s-1}} x_s^{m_s +1}.$$
The least term of this polynomial with respect to the lexicographic
order is the monomial $x_1^{m_1} \cdots x_s^{m_s}$, which has the coefficient~$1$. 
This interpretation leads us to introduce the following notion.

Let~$\prec$ be a monomial order on the polynomial ring 
$R[x_1,x_2,\ldots]$ with infinitely many variables.  For every polynomial
$f$ we write $\ini_\prec(f)$ for the least term of $f$ with respect to
$\prec$.  Let $R[X] = R[x_1 \upto x_s]$. We call $a_1 \upto a_s \in R$ a {\em
  dependent sequence} with respect to~$\prec$ if there exists $f \in
R[X]$ vanishing at $a_1 \upto a_s$ such that the coefficient of
$\ini_\prec(f)$ is invertible. Otherwise, $a_1 \upto a_s$ is called an {\em independent
  sequence} with respect to~$\prec$. \par
 
Using this notion, we can reformulate Lombardi's result as $\dim R <
s$ if and only if  every sequence of elements $a_1 \upto a_s$ in $R$
is dependent with respect to the lexicographic order.  Out of
this reformulation arises the question whether one can replace the
lexicographical monomial order by other monomial orders.  
The proof of Lombardi  does not reveal how one can relate an arbitrary
monomial order to the Krull dimension of the ring. We will give a
positive answer to this question by proving
that  $\dim R$ is the supremum of the length of independent sequences for
 an arbitrary monomial order. This follows from \tref{monomial
   order 1} of this paper, which in fact strengthens the above statement.
 As an immediate consequence,  we obtain other algebraic identities between elements of $R$ than in Lombardi's result.
Although our results are not essentially computational, the independence conditions can often be treated by computer calculations. For instance, using a short program
written in MAGMA \cite{Bosma:1997}, the first author tested millions of examples which led to the conjecture 
that the above question has a positive answer \cite{Kemper:2011}.  
The proof combines techniques of Gr\"obner basis theory and the theory of associated graded rings of filtrations. 
It has led to other notions of independent sequences which are of independent interest, as we shall see below. \par

Our idea is to replace the monomial order $\prec$ by a weighted degree on the monomials.
Given  an infinite sequence $\w$ of positive integers $w_1,w_2,\ldots$, 
we may consider $R[x_1,x_2,\ldots]$ as a weighted graded ring with $\deg x_i = w_i$, $i = 1,2,\ldots$.
For every polynomial~$f$,  we write $\ini_\w(f)$ for the weighted homogeneous part of~$f$ of least degree.
We call $a_1 \upto a_s \in R$ a {\em weighted independent sequence} with respect to $\w$ 
if every coefficient of $\ini_\w(f)$ is not invertible for all polynomials $f \in R[X]$ vanishing at $a_1 \upto a_s$.
Otherwise, $a_1 \upto a_s$ is called a {\em weighted dependent sequence} with respect to $\w$.
We will see that if $R$ is a local ring and $w_i = 1$ for all $i$, the sequence $a_1 \upto a_s$ is weighted independent if and only if 
the elements $a_1 \upto a_s$ are analytically independent, a basic notion in the theory of local rings.
That is the reason why we use the terminology independent sequence for the above notions. \par

Let $Q = (x_1-a_1 \upto x_s-a_s)$ be the ideal of polynomials of $R[X]$ vanishing at $a_1 \upto a_s$.
Let $\ini_\prec(Q)$ and  $\ini_\w(Q)$ denote the ideals of $R[X]$
generated by the polynomials $\ini_\prec(f)$ and $\ini_\w(f)$, $f \in Q$. We want to find a weight sequence $\w$ such that $\ini_\prec(Q) = \ini_\w(Q)$. It is well known in Gr\"obner basis theory that this can be done if $\ini_\prec(Q)$ and  $\ini_\w(Q)$ were the largest term or the part of largest degree of $f$.
In our setting we can solve this problem only if $\prec\ $ is Noetherian,  that is,  if every monomial has only a
finite number of smaller monomials. In this case, $a_1,...,a_s$ is independent with respect to $\w$ if and only if it is independent with respect to $\prec\ $. If $\prec\ $ is not Noetherian, we can still find a Noetherian monomial order $\prec'$ such that 
if $a_1,...,a_s$ is independent with respect to $\prec$, then $a_1a_i,...,a_sa_i$ is independent with respect to $\prec'$ for some index $i$.
By this way, we can reduce our investigation on the length of independent sequences to the weighted graded case. \par

We shall see that for every weight sequence $\w$, $\ini_\w(Q)$ is the defining ideal of  the associated graded ring of certain filtration of $R$. 
Using properties of this associated graded ring we can show that the length of a weighted independent sequence is bounded above by $\dim R$, and that $a_1 \upto a_s$ is a weighted independent sequence if $\height(a_1 \upto a_s) = s$.
From this it follows that $\dim R$ is the supremum of the length of independent sequences with respect to $\w$. This is formulated in more detail in \tref{weight1} of this paper.
Furthermore, we can also show that $\dim R/\cup_{n \ge 1}(0:J^n)$ is the supremum of the length
of weighted independent sequences in a given ideal $J$. If $R$ is a local ring, 
this gives a characterization for the maximum number of analytically independent elements in $J$. \par
 
Since our results for independent sequences with respect to a monomial
order and for weighted independent sequences are analogous, 
one may ask whether there is a common generalization.
We shall see that there is a natural class of binary relations on the monomials
which cover both monomial orders and weighted degrees and 
for which the modified statements of the above results still hold. 
We call such a relation a {\em monomial preorder}. 
The key point is to show that a monomial preorder~$\prec$ can be
approximated by a weighted degree sequence~$\w$.
This is somewhat tricky because~$\w$ has to be chosen such that incomparable monomials with respect to~$\prec$ 
have the same weighted degree. Since monomial preorders are not as strict as monomial orders, 
these results may find applications in computational problems. \par
 
For an algebra over a ring, we can extend the definition of independent sequences to give a generalization of the transcendence degree.
Let $A$ be an algebra over $R$. Given  a  monomial preorder~$\prec$, we say that
a sequence $a_1 \upto a_s$ of elements of $A$ is independent over $R$ 
with respect to~$\prec$ if for every polynomial $f \in R[X]$ 
vanishing at $a_1 \upto a_n$, no coefficient of $\ini_\prec(f)$ is invertible in $R$. 
If $R$ is a field, this is just the usual notion of algebraic independence.  
In general, $\dim A$ is not the supremum of the length of independent sequences over $R$. 
However, if $R$ is a Jacobson ring and $A$ a subfinite $R$-algebra, that is, a subalgebra of a finitely generated  $R$-algebra,
we show that $\dim A$ is the supremum of the length of independent sequences with respect to $\prec$. 
So we obtain a generalization of the fundamental result that the
transcendence degree of a finitely generated algebra over a field
equals its Krull dimension.  Our result has the interesting
consequence that the Krull dimension cannot increase if one passes
from a subfinite algebra over a Noetherian Jacobson ring to a
subalgebra. For instance, if $H \subseteq G \subseteq \Aut(A)$ are groups of automorphisms of a finitely generated $\ZZ$-algebra $A$, then  
  \[
  \dim\left(A^G\right) \le \dim\left(A^H\right),
  \]
  even though the invariant rings need not be finitely generated. 
We also show that the above properties characterize Jacobson rings. \par

The paper is organized as follows. In Sections 1 and 2 we investigate weighted independent sequences and independent sequences with respect to a monomial order. The extensions of these notions for monomial preorders and for algebras over a Jacobson ring will be treated in Sections 3 and 4, respectively. 
\par

We would like to mention that there exists an earlier
version of this paper, titled ``The Transcendence Degree over a Ring'' and authored by the
first author~\cite{Kemper:2011}. This earlier version will not be published since its
results have merged into the present version.

The authors wish to thank Peter Heinig for bringing Coquand and
Lombardi's article \cite{Coquand-Lombardi:05} to their attention,
which initiated our investigation. They also thank Jos\'e Giral, Shiro Goto, J\"urgen Kl\"uners, Gerhard Pfister, Lorenzo Robbiano, Keiichi Watanabe for sharing their expertise, and the referee for pointing out that Lemma 1.6 can be found in \cite{B-H}. The second author is grateful to the Mathematical Sciences Research Institute at Berkeley for its support
to his participation to the Program Commutative Algebra  2012-2013, when part of this paper was written down. He is supported by a grant of the National Foundation for Sciences and Technology Development of Vietnam.

\section{Weighted independent sequences} \label{1sWeighted}

In this section we will prove some basic properties of weighted
independent sequences and our aim is to show that the Krull dimension
is the supremum of the length of weighted independent sequences.

Throughout this paper, let $R$ be a Noetherian ring. 
Let $a_1 \upto a_s$ be a sequence of nonzero elements of $R$, which are not invertible.
Note that an element of $R$ is weighted dependent if it is zero or
invertible. \par

First, we shall see that weighted independent sequences are a generalization of analytically independent elements.
Recall that if $R$ is a local ring, the elements $a_1,\ldots,a_s$ are called {\em analytically independent} if  every homogeneous polynomial vanishing at $a_1 \upto a_s$ has all its coefficients in the maximal ideal, which means that they are not invertible. \par

Let $\w = 1,1,\ldots$, the weight sequence with all $w_i = 1$.  The weighted degree in this case is the usual degree.  
Hence $\ini_\w(f)$ is the homogeneous part of smallest degree of a
polynomial~$f$.
Thus, $a_1 \upto a_s$ is analytically dependent if there exists a homogeneous polynomial vanishing at $a_1 \upto a_s$
which has an invertible coefficient.

\begin{ex} \label{1exZ}
 Let $a,b$ be two arbitrary integers. Since the greatest common divisor of $a^2$ and $b^2$
  divides the product $a b$, there exist $c,d \in \ZZ$ such
  that $a b = c a^2 + d b^2$. This relation shows that $a,b$ is a
  weighted dependent sequence with respect to $\w = 1,1,\ldots$
\end{ex}
      
Set $R[X] = R[x_1 \upto x_s]$. Let $f \in R[X]$ be an arbitrary
polynomial vanishing at $a_1 \upto a_s$ and $g = \ini_\w(f)$, where
$\w = 1,1,\ldots$. Write every term $u$ of $f$ with $\deg u > \deg g$
in the form $u = hv$, where $v$ is a monomial with $\deg v = \deg g$,
and replace $u$ by the term $h(a_1 \upto a_s)v$.  Then we obtain a homogeneous polynomial of the
form $g + a_1g_1 + \cdots + a_sg_s$ vanishing at $a_1 \upto a_s$.
If $R$ is a local ring, the coefficients of $g$ are
not invertible if and only if the coefficients of $g + a_1g_1 + \cdots
+ a_sg_s$ are not invertible. Hence $a_1 \upto a_s$
is a weighted independent sequence if and only if $a_1 \upto a_s$ are analytically independent. \par
 
Unlike analytically independent elements, the notion of weighted independent sequences depends on the order of the elements if the weight sequence $\w$ contains some distinct numbers.

\begin{ex} \label{not permute}%
  Let $R = K[u,v]$ be a polynomial ring in two indeterminates over a
  ring $K$.  The sequence $uv, v$ is dependent with respect to the
  weights $1,2$ because $x_1 - ux_2$ vanishes at $uv, v$ and $\ini_\w(x_1 - ux_2) =
  x_1$.  On the other hand, the sequence $v,uv$ is independent with
  respect to the same weights. To see this let $f= (x_1-v)g +
  (x_2-uv)h$ be an arbitrary polynomial of $R[x_1,x_2]$ vanishing at
  $v,uv$.  If $vg +uvh \neq 0$, $\ini_\w(f) = -\ini_\w(vg +uvh)$,
  whose coefficients are divided by $v$, hence not invertible.  If $vg
  +uvh = 0$, $g = uh$ and $\ini_\w(f) = \ini_\w(x_1uh +x_2h) =
  \ini_\w(x_1uh)$ since $\deg x_1 = 1 < 2 = \deg x_2$.  All
  coefficients of $\ini_\w(x_1uh)$ are divided by $u$, hence not
  invertible.
\end{ex}

Let $\w$ be an arbitrary weight sequence. 
Let $Q = (x_1-a_1 \upto x_s-a_s)$, the ideal of polynomials of $R[X]$ vanishing at $a_1 \upto a_s$.
Let $C$ be the set of the coefficients of all polynomials $\ini_\w(f)$, $f \in Q$. 
It is easy to see that $C$ is an ideal. 
Therefore, $a_1,\ldots,a_s$ is a weighted independent sequence with respect to $\w$
if and only if $C$ is a proper ideal of $R$.  Using this characterization, we obtain 
the following property of weighted independent sequences under localization.

\begin{prop} \label{local}%
  The sequence $a_1,\ldots,a_s$ is weighted independent if and only if
  there is a prime $P$ of $R$ such that $a_1,\ldots,a_s$ is weighted
  independent in $R_P$.
\end{prop}

\begin{proof}
  If $a_1,\ldots,a_s$ is a weighted independent sequence, then $C$ is
  contained in a maximal ideal $P$ of $R$.  Since $Q_P$ is the ideal
  of the polynomials in $R_P[X]$ vanishing at $a_1,\ldots,a_s$, $C_P$
  is the set of the coefficients of all polynomials $\ini_\w(f)$, $f
  \in Q_P$. Since $C_P$ is a proper ideal of $R_P$, $a_1,\ldots,a_s$
  is a weighted independent sequence in $R_P$.

  Conversely, if $a_1,\ldots,a_s$ is a weighted independent sequence
  in $R_P$ for some prime $P$ of $R$, then $C_P$ is a proper ideal and
  so is $C$, too.  Therefore, $a_1,\ldots,a_s$ is a weighted
  independent sequence in $R$.
\end{proof}

Let $\ini_\w(Q)$ denote the ideal in $R[X]$ generated by the polynomials $\ini_\w(f)$, $f
\in Q$. Then $C$ is also the set of the coefficients of all polynomials in $\ini_\w(Q)$.
Therefore, weighted independence is a property of $\ini_\w(Q)$.
We shall see that $R[X]/\ini_\w(Q)$ is isomorphic to the associated graded ring of certain filtration of $R$.\par
 
Let $S$ denote the subring $R[a_1t^{w_1},\ldots,a_st^{w_s},t^{-1}]$ of
the Laurent polynomial ring $R[t,t^{-1}]$.  Since $S$ is a graded
subring of $R[t, t^{-1}]$, we may write $S = \oplus_{n \in
  \ZZ}I_nt^n$. It is easy to see that
\begin{equation} \label{filtration}
I_n = \sum_{\scriptsize \begin{array}{c} m_1w_1+\cdots+m_sw_s  \ge n\\
m_1,\ldots,m_s \ge 0\end{array}} a_1^{m_1}\cdots
a_s^{m_s}R
\end{equation}
for $n \ge 0$ and $I_n = R$ for $n < 0$. The ideals $I_n$, $n \ge 0$,
form a filtration of $R$. In the case $w_1 = \cdots = w_s =1$, we have
$I_n = I^n$, where $I := (a_1 \upto a_s)$. 
So we may consider $S$ as the extended Rees algebra of this filtration.
\par

Let $G = S/t^{-1}S$. Then $G \cong \oplus_{n \ge 0}I_n/I_{n+1}$. In
other words, $G$ is the associated graded ring of the above
filtration.

\begin{lemma} \label{presentation}
$G \cong R[X]/\ini_\w(Q)$.
\end{lemma}

\begin{proof}
  Let $y$ be a new variable and consider the polynomial ring $R[X,y]$
  as weighted graded with $\deg x_i = w_i$ and $\deg y = -1$.  Then we
  have a natural graded map $R[X,y] \to S$, which sends $x_i$ to
  $a_it^{w_i}$, $i = 1,\ldots,s$, and $y$ to $t^{-1}$.  Let $\Im$ denote
  the kernel of this map. Then $S \cong R[X,y]/\Im$, hence
  \[
  G \cong R[X,y]/(\Im,y) \cong R[X]/\left((\Im,y) \cap R[X]\right).
  \]
  It remains to show that $(\Im,y) \cap R[X] = \ini_\w(Q).$ \par

  Let $g$ be an arbitrary element of $(\Im,y) \cap R[X]$. Then $g = F
  + Hy$ for some polynomials $F \in \Im$ and $H \in R[X,y]$. Without
  loss of generality we may assume that $g$ is nonzero and that $g$
  and $F$ are weighted homogeneous. Then $F$ has the form $F = g +
  g_1y + \cdots + g_ny^n$, where $g_i$ is a weighted homogeneous
  polynomial of $R[X]$ with $\deg g_i = \deg g +i$, $i =
  1,\ldots,n$. Set $f = g + g_1 + \cdots + g_n$. We have
  $f(a_1,\ldots,a_s)t^{\deg g} = F(a_1t^{w_1},\ldots,a_st^{w_s},
  t^{-1}) = 0$.  Therefore, $f(a_1,\ldots,a_s) = 0$ and hence $f \in
  Q$. Since $g = \ini_\w(f)$, $g \in \ini_\w(Q)$. \par

  Conversely, every polynomial $f \in Q$ can be written in the form $f
  = g + g_1 + \cdots + g_n$, where $g = \ini_\w(f)$ and $g_i$ is a
  weighted homogeneous polynomial with $\deg g_i = \deg g + i$, $i =
  1,\ldots,n$.  Set $F = g + g_1y + \cdots + g_ny^n$. Then
  $F(a_1t^{w_1},\ldots,a_st^{w_s}, t^{-1}) = f(a_1,\ldots,a_s)t^{\deg g} =
  0$. Therefore $F \in \Im$ and hence 
$\ini_\w(f) = F - (g_1 + \cdots   + g_ny^{n-1})y \in (\Im,y).$
\end{proof}

\begin{cor} \label{bound}
If $a_1,...,a_s$ is a weighted independent sequence, then $s \le \dim G$. 
\end{cor}

\begin{proof}
Since $\ini_\w(Q) \subseteq CR[X]$, there is a surjective map $R[X]/\ini_\w(Q) \to
R[X]/CR[X]$. Since $C$ is a proper ideal of $R$, $s \le  \dim R[X]/ CR[X]$ because
$R[X]/CR[X] \cong (R/C)[X]$, the polynomial ring
in $s$ variables over $R/C$. Thus,  $s \le \dim R[X]/\ini_\w(Q) = \dim G$.
\end{proof}

The following formula for $\dim G$ follows from a general formula for the dimension of the associated graded ring of a filtration \cite[Theorem 4.5.6(b)]{B-H}. This formula is a generalization of the well-known fact that $\dim G = \dim R$ 
if $R$ is a local ring and $G$ is the associated graded ring of an ideal (see Matsumura~\cite[Theorem~15.7]{Matsumura:86} or Eisenbud~\cite[Excercise 13.8]{eis}).

\begin{lemma} \label{associated} 
Let $I = (a_1 \upto a_s)$. Then
$$\dim G = \sup\{\height P \mid P \supseteq I\ \text{is a prime of}\ R\}.$$  
\end{lemma}

As a consequence, we always have $\dim G \le \dim R$. 
Together with  \cref{bound},  this implies that the length of a weighted independent sequence cannot exceed $\dim R$. 
Now we will show that there exist weighted independent sequences
of length $\height P$ for any maximal prime $P$ of $R$. \par

Let $\bigheight(I)$ denote the {\em big height} of $I$, that is, the
maximum height of the minimal primes over $I$.  

\begin{prop} \label{parameter}%
  Let $a_1,\ldots,a_s$ be elements of $R$ such that
  $\bigheight(a_1,\ldots,a_s) = s$. Then $a_1,\ldots,a_s$ is a
  weighted independent sequence with respect to every weight sequence
  $\w$.
\end{prop}

\begin{proof}
  Let $P$ be a minimal prime of $I = (a_1,\ldots,a_s)$ with $\height P
  = s$. By \pref{local}, $a_1,\ldots,a_s$ is a weighted
  independent sequence in $R$ if $a_1,\ldots,a_s$ is a weighted
  independent sequence in $R_P$. Therefore, we may assume that $R$ is
  a local ring and $a_1,\ldots,a_s$ is a system of parameters in
  $R$. In this case, $\dim G = s$ by \lref{associated}.

  Let~$\m$ be the maximal ideal of $R$.  There exists an integer~$r$
  such that $\m^r \subseteq I$. Since $I_1 = I$, $\m^r I_n \subseteq
  I_{n+1}$ for all $n$. Therefore, $\m^r G = 0$. Hence $\dim G/\m G =
  \dim G = s$.  Let $k = R/\m$. By \lref{presentation}, $G/\m G =
  R[X]/(\ini_\w(Q),\m) = k[X]/J$ for some ideal $J$ of $k[X]$. If
  $a_1,\ldots,a_s$ were weighted dependent, there would be a
  polynomial in $\ini_\w(Q)$ which has a coefficient not in $\m$,
  implying $J \neq 0$ and the contradiction $\dim (G/\m G) \le s-1$.
\end{proof}

Summing up, we obtain the following results on the Krull
dimension in terms of weighted independent sequences.

\begin{theorem} \label{weight1}%
  Let $R$ be a Noetherian ring and~$s$ a positive integer.
  \begin{enumerate}
  \item \label{weight1A} If $s \le \dim R$, there exists a sequence
    $a_1 \upto a_s \in R$ that is weighted independent with respect to
    every weight sequence.
  \item \label{weight1B} If $s > \dim R$, every sequence $a_1 \upto
    a_s \in R$ is weighted dependent with respect to every weight
    sequence.
  \end{enumerate}
\end{theorem}

\begin{proof}
  If $s \le \dim R$, there exists a prime $P$ in $R$ of height~$s$.  
It is a standard fact that there exists a sequence $a_1 \upto a_s \in P$ such that $P$ is a minimal prime of $(a_1 \upto a_s)$. 
Hence \eqref{weight1A} follows from \pref{parameter}. 
If $s > \dim R$, then $s > \dim G$ by Lemma \ref{associated}.
Hence \eqref{weight1B} follows from Corollary \ref{bound}.
\end{proof}

As a consequence, $\dim R$ is the supremum of the length of 
weighted independent sequences with respect to an arbitrary weight sequence. \par

\begin{remark*}
A maximal weighted independent sequence need not to have length $\dim R$.
To see that we consider a Noetherian ring that has a maximal ideal $P = (a_1,...,a_s)$ with $s = \height P < \dim R$. 
By Proposition \ref{parameter}, $a_1,\ldots,a_s$ is weighted independent with respect to every weight sequence $\w$.
It is maximal because any extended sequence $a_1,...,a_{s+1}$ with $a_{s+1} \not\in P$ is weighted dependent.
This follows from the fact $R = (a_1,...,a_{s+1})$, which implies that there is a polynomial $f$ of the form $1+ c_1x_1 + \cdots  + c_{s+1}x_{s+1}$ vanishing at  $a_1,...,a_{s+1}$ with $\ini_\w(f) = 1$.
\end{remark*}

Similarly, we can study weighted independent sequences in a given ideal
$J$ of $R$. Let $0:J^\infty = \cup_{m \ge
  0}0:J^m$. Note that $0:J^\infty$ is the intersection of all primary
components of the zero-ideal $0_R$ whose associated primes do not
contain $J$ and that $0:J^\infty = 0:J^m$ for $m$ large enough.

\begin{theorem} \label{weight2}%
  For every ideal $J \subseteq R$,  $\dim R/0:J^\infty$ is the supremum of the length
  of weighted independent sequences in $J$ with respect to an arbitrary weight sequence.
\end{theorem}

\begin{proof}
  Let $P$ be a maximal prime of $R/0:J^\infty$ and $s = \height P$.
  Using Proposition \ref{parameter} we can find elements $a_1,\ldots,a_d$ in $R$ such that their
  residue classes in $R/0:J^\infty$ is a weighted independent sequence. 
 Choose $c \in J$ such that $c$ is not contained in any associated prime of $0_R$
  not containing $J$. Then $0:c^\infty = 0:J^\infty$. We claim that $a_1c^{w_1},\ldots,a_sc^{w_s}$ is a weighted
  independent sequence. To see this let $f$ be a polynomial in
  $R[X]$ vanishing at $a_1c^{w_1},\ldots,a_sc^{w_s}$ and
  $r = \deg \ini_\w(f)$. Write $f$ in the form $f = \ini_\w(f) + g_1 +
  \cdots + g_n$, where $g_i$ is a weighted homogeneous polynomial of
  degree $r + i$, $i = 1,\ldots,n$. Then
  \begin{align*}
  & f(a_1c^{w_1},\ldots,a_sc^{w_s}) \\
& = c^r\ini_\w(f)(a_1,\ldots,a_s) +
  c^{r+1}g_1(a_1,\ldots,a_s) + \cdots + c^{r+n}g_n(a_1,\ldots,a_s) =
  0.
  \end{align*}
  Therefore, if we put $h = \ini_\w(f) + cg_1 + \cdots + c^ng_n$, then
  $h(a_1,\ldots,a_s) \in 0:c^d \subseteq 0:J^\infty$ and $\ini_\w(h) =
  \ini_\w(f)$.  By the choice of $a_1,\ldots,a_s$, the coefficients of
  $\ini_\w(f)$ cannot not be invertible. This shows the existence of a
  weighted independent sequence of length $s$ in $J$. Hence $\dim R/0:J^\infty$ is less than or equal to
the supremum of the length of weighted independent sequences in $J$. 
\par

  Now we will show that $s \le \dim R/0:J^\infty$ for any weighted independent
  sequence $a_1,\ldots,a_s$ in $J$.  Let $m$ be a positive number such
  that $0:J^\infty = 0:J^m$. Then $(0:J^\infty)a_i^m = 0$, $i =   1,\ldots,s$.  
This implies $(0:J^\infty)x_i^m \subseteq \ini_\w(Q)$.  Hence $0:J^\infty \subseteq C$, 
where $C$ is the ideal of the coefficients of polynomials in $\ini_\w(Q)$.
Let~$\m$ be a maximal ideal of $R$ containing $C$. 
Then $\ini_\w(Q) + (0:J^\infty)R[X] \subseteq \m R[X]$. 
By Lemma \ref{associated} there  is a surjective map $G/(0:J^\infty)G \to R[X]/\m R[X]$. 
From this it   follows that $s = \dim R[X]/\m R[X] \le \dim G/(0:J^\infty)G$. We
  have  
  \[
  G/(0:J^\infty)G = \oplus_{n \ge 0}I_n/\big((0:J^\infty)I_n+I_{n+1}\big),
  \]
  and we will compare this with the ring
  \[
  G' := \oplus_{n \ge 0}\big(I_{n}+ (0:I^\infty)\big)/\big(I_{n+1}+
  (0:I^\infty)\big),
  \]
  which is the associated graded ring of $R/0:J^\infty$ with respect
  to the filtration $\big(I_n+(0:J^\infty)\big)/(0:J^\infty)$, $n \ge
  0$.  Note that $\dim G' \le \dim R/0:J^\infty$ by Lemma
  \ref{associated}.  Then $s \le \dim R/0:J^\infty$ if we can show that $\dim
  G/(0:J^\infty)G = \dim G'$.

  Let $w_{\max} := \max \{w_i|\ i = 1,\ldots,s\}$. By \eref{filtration} we have
  $I_n \subseteq I^m$ for $n \ge mw_{\max}$.   This implies $(0:J^\infty)I_n \subseteq (0:J^\infty)I^m = 0$.  
Using Artin-Rees lemma we can also show that
  $(0:J^\infty) \cap I_n = 0$ for $n$ large enough.  Thus, there
  exists a positive number $r$ such that $(0:J^\infty)I_n =
  (0:J^\infty) \cap I_n = 0$ for $n \ge r$.  This relation implies
  \[
  I_n/\big((0:J^\infty)I_n+I_{n+1}\big) = I_n/\big((0:J^\infty)\cap
  I_n + I_{n+1}\big) \cong \big(I_n + (0:J^\infty)\big)/\big(I_{n+1} +
  (0:J^\infty)\big).
  \]
  Hence
  \[\bigoplus_{n \ge 0} I_{nr}/\big((0:J^\infty)I_{nr}+I_{nr+1}\big)
  \cong \bigoplus_{n \ge 0}\big(I_{nr}+
  (0:J^\infty)\big)/\big(I_{nr+1}+ (0:J^\infty)\big).
  \]
  The graded rings on both sides are Veronese subrings of
  $G/(0:J^\infty)G$ and $G'$, respectively. Since the dimension of a
  Veronese subring is the same as of the original ring, we get $\dim
  G/(0:J^\infty)G = \dim G'$, as required. 
\end{proof}

Theorem  \ref{weight2} has the following immediate consequence.

\begin{cor} 
Let $R$ be a local ring and $J$ an ideal of $R$. 
Then $\dim R/0:J^\infty$ is the maximum number of analytically independent elements in $J$.
\end{cor}

This result seems to be new though there was a
general (but complicated) formula for the maximum number of $\frak
a$-independent elements in $J$, where $\frak a$ is an ideal containing
$J$ (see~\cite{Bjork:82}, \cite {Trung:85}). Recall that the elements
$a_1 \upto a_s$ are called $\frak a$-independent if every homogeneous
form in $R[X]$ vanishing at $a_1 \upto a_s$ has all its coefficients
in $\frak a$.  This notion was introduced by Valla
\cite{Valla:70}.

\section{Independent sequences with respect to a monomial
  order} \label{2sOrder}

In this section we will show how to approximate a monomial order by a
weighted degree and we will prove that the Krull dimension is the
supremum of the length of independent
sequences with respect to an arbitrary monomial order. \par

Let $a_1,\ldots,a_s$ be elements of a Noetherian ring $R$. 
Recall that $a_1 \upto a_s$ is a {\em   dependent sequence} 
with respect to a monomial order~$\prec$ if there exists $f \in
R[x_1,...,x_s]$ vanishing at $a_1 \upto a_s$ such that the coefficient of
$\ini_\prec(f)$ is invertible. Otherwise, $a_1 \upto a_s$ is called an {\em independent
  sequence} with respect to~$\prec$. 

The following example suggests that dependence with respect to a
monomial order is more subtle than weighted dependence.

\begin{ex} \label{2exZ}%
  Let $R = \ZZ$ and let $\prec$ be the lexicographic order with $x_1
  \succ x_2$. Clearly the single elements that are dependent with
  respect to~$\prec$ are~$0$ and the invertible elements. We claim
  that a sequence of two arbitrary integers $a,b$ is always dependent
  with respect to~$\prec$. The relation $a b = c a^2 + d b^2$ found in
  \exref{1exZ} does not show the dependence, so we have to argue in a
  different way. We may assume~$a$ and~$b$ to be nonzero and write
  \[
  a = \pm \prod_{i=1}^r p_i^{d_i} \quad \text{and} \quad b = \pm
  \prod_{i=1}^r p_i^{e_i},
  \]
  where the $p_i$ are pairwise distinct prime numbers and $d_i,e_i \in
  \NN_0$. Choose $n \in \NN_0$ such that $n \ge d_i/e_i$ for all~$i$
  with $e_i > 0$. Then
  \[
  \gcd(a,b^{n+1}) = \prod_{i=1}^r p_i^{\min\{d_i,(n+1) e_i\}} \quad
  \text{divides} \quad \prod_{i=1}^r p_i^{n e_i} = b^n,
  \]
  so there exist $c,d \in \ZZ$ such that $b^n = c a + d b^{n+1}$.
  Since the least term of $f = x_2^n - c x_1 - d x_2^{n+1}$ is $x_2^n$
  this relation shows that $a,b$ are dependent, as claimed.

  The argument can easily be adapted to any Dedekind domain.
\end{ex}

It is easy to see that the notion of independent sequence depends on
the order of the elements. For instance, the sequence $uv,v$ of
\exref{not permute} is independent with respect to the lexicographic
order, while $v,uv$ is not by using the same arguments. \par

Set $R[X] = R[x_1 \upto x_s]$ and $Q =
(x_1-a_1,\ldots,x_s-a_s)$, the ideal of all polynomials of $R[X]$
vanishing at $a_1,\ldots,a_s$.  Let $\ini_\prec(Q)$ denote the ideal
generated by the terms $\ini_\prec(f)$, $f \in Q$.  One may ask
whether there exists a weight sequence~$\w$ such that $\ini_\w(Q) =
\ini_\prec(Q)$. For this will imply that $a_1,\ldots,a_s$ is an
independent sequence with respect to~$\prec$ if and only if it is a
weighted independent sequence with respect to~$\w$. \par

To study this problem we need the following result in Gr\"obner basis theory.

\begin{lemma}[{see Eisenbud \cite[Exercise 15.12]{eis},
    \cite[Exercise~9.2(b)]{Kemper.Comalg}}] \label{approximate}%
  Let $\mathcal M$ be a finite set of polynomials. Then there exists a weight
  sequence $\w$ such that $\ini_\prec(f) = \ini_\w(f)$ for all $f \in
  \mathcal M$.
\end{lemma}

We call~$\prec$ a {\em Noetherian monomial order} if for every
monomial $f \in R[X]$ there are only finitely many monomials $g \in
R[X]$ with $g \prec f$. This class of monomial orders is rather
large. For instance, every monomial order that first compares the
(weighted) degree of the monomials is Noetherian.

\begin{prop} \label{compare}%
  For every ideal $\Im$ of $R[X]$, there exists a weight sequence $\w$
  such that $\ini_\prec(\Im) \subseteq \ini_\w(\Im)$. If~$\prec\ $ is
  Noetherian, $\w$ can be chosen such that $\ini_\prec(\Im) =
  \ini_\w(\Im)$.
\end{prop}

\begin{proof}
  Choose $g_1,\ldots,g_r \in \Im$ such that $\ini_\prec(\Im) =
  \big(\ini_\prec(g_1),\ldots, \ini_\prec(g_r)\big)$.  From
  \lref{approximate} it follows that there exists a weight sequence
  $\w$ such that $\ini_\prec(g_i) = \ini_\w(g_i)$ for all $i =
  1,\ldots,r$. This implies the first assertion:
  \[
  \ini_\prec(\Im) = \big(\ini_\w(g_1),\ldots,\ini_\w(g_r)\big) \subseteq
  \ini_\w(\Im).
  \]
  Now we will assume that~$\prec$ is Noetherian and prove equality.
  By way of contradiction, assume that there exists a polynomial $f
  \in \Im$ such that $\ini_\w(f) \not\in \ini_\prec(\Im)$.  Choose $f$
  such that $\ini_\w(f)$ has the least possible number of terms. For
  every $g \in R[X]$ we have $\ini_\prec(g) \preceq
  \ini_\prec(\ini_\w(g))$, so $\ini_\prec(\ini_\w(f))$ is an upper
  bound for all initial terms of polynomials~$g$ with $\ini_\w(g) =
  \ini_\w(f)$. By the assumption on the monomial order, we can
  therefore choose~$f$ such that for all $g \in \Im$,
  \begin{equation} \label{2eqMax}%
    \ini_\w(g) = \ini_\w(f) \quad \text{implies} \quad \ini_\prec(g)
    \preceq \ini_\prec(f).
  \end{equation}
  Since $\ini_\prec(f) \in \ini_\prec(\Im)$, we have $\ini_\prec(f) =
  h_1\ini_\prec(g_1) + \cdots + h_r \ini_\prec(g_r)$ for some
  polynomials $h_1,\ldots,h_r$. By deleting some terms of the $h_i$,
  we may assume that either $h_i = 0$ or $h_i$ is a term such that
  $h_i \ini_\prec(g_i)$ and $\ini_\prec(f)$ are $R$-multiples of the
  same monomial. Set $h = h_1g_1 + \cdots + h_rg_r \in \Im$. Then
  \begin{equation} \label{2eqH}
    \ini_\prec(f) = \ini_\prec(h) = \ini_\w(h),
  \end{equation}
  where the second equality follows from $\ini_\prec(g_i) =
  \ini_\w(g_i)$. For $g := f - h \in \Im$, this implies $\ini_\prec(g)
  \succ \ini_\prec(f)$, so $\ini_\w(g) \ne \ini_\w(f)$
  by~\eqref{2eqMax}. We also have $\ini_\w(h) \ne \ini_\w(f)$ because
  otherwise $\ini_\w(f) = \ini_\prec(f) \in \ini_\prec(I)$
  by~\eqref{2eqH}. For the weighted degrees we have the inequality
  \[
  \deg\left(\ini_\w(f)\right) \le \deg\left(\ini_\prec(f)\right) =
  \deg\left(\ini_\w(h)\right).
  \]
  In combination with $\ini_\w(g) \ne \ini_\w(f) \ne \ini_\w(h)$, this
  implies that $\ini_\w(f)$, $\ini_\w(h)$, and $\ini_\w(g)$ all have
  the same degree. So $\ini_\w(g) = \ini_\w(f) -
  \ini_\w(h)$. By~\eqref{2eqH}, subtracting $\ini_\w(h)$ from
  $\ini_\w(f)$ removes the initial term of $\ini_\w(f)$ but leaves all
  other terms unchanged. So $\ini_\w(g)$ has fewer terms than
  $\ini_\w(f)$, and because of the choice of~$f$ we conclude
  $\ini_\w(g) \in \ini_\prec(\Im)$. But since $\ini_\w(h) \in
  \ini_\prec(\Im)$ by~\eqref{2eqH}, this implies $\ini_\w(f) \in
  \ini_\prec(\Im)$, a contradiction.
\end{proof}

We do not know whether the Noetherian hypothesis is really necessary
for the second assertion of \pref{compare}.
  
\begin{remark*}
  For a polynomial $f \in R[X]$, we can also consider the leading term
  $\LT_\prec(f)$ and the weighted homogeneous part of highest degree,
  i.e., the {\em leading form} $\LF_\w(f)$. This defines
  $\LT_\prec(\Im)$ and $\LF_\w(\Im)$ for an ideal $\Im \subseteq R[X]$.
  Then \pref{compare} remains correct without the Noetherian
  hypothesis if we substitute $\ini_\prec$ by $\LT_\prec$ and
  $\ini_\w$ by $\LF_\w$.  This is a well known result in Gr\"obner
  basis theory (see Eisenbud~\cite[Proposition 15.16]{eis} or
  \cite[Exercise~9.2(c)]{Kemper.Comalg}).
\end{remark*}

Proposition \ref{compare} implies the following relationship between
weighted independent sequences and independent sequences with respect
to monomial orders.

\begin{cor} \label{Noetherian}%
  Let $a_1 \upto a_s$ be a sequence of elements in $R$.
  \begin{enumerate}
  \item \label{NoetherianA} If the sequence is weighted independent
    with respect to every weight sequence, it is independent with
    respect to every monomial order.
  \item \label{NoetherianB} If the sequence is weighted dependent with
    respect to every weight sequence, it is dependent with respect to
    every Noetherian monomial order.
  \end{enumerate}
\end{cor}

\begin{proof}
  By \pref{compare} there exists a weight sequence~$\w$ such that
  $\ini_\prec(Q) \subseteq \ini_\w(Q)$, with equality if~$\prec$ is
  Noetherian. Under the hypothesis of~\eqref{NoetherianA} there exists
  a maximal ideal $\m \subset R$ such that $\ini_\w(Q) \subseteq \m
  R[X]$, so $\ini_\prec(Q) \subseteq \m R[X]$. This shows that the
  sequence is independent with respect to~$\prec$.

  Under the hypothesis of~\eqref{NoetherianB}, the ideal $C$ of
  coefficients of polynomials in $\ini_\w(Q)$ is $R$. Since
  $\ini_\w(Q) = \ini_\prec(Q)$, this implies that the sequence is
  dependent with respect to~$\prec$.
\end{proof}

We will get rid of the Noetherian hypothesis on a monomial order by
showing that an independent sequence with respect to an arbitrary
monomial order can be converted to an independent sequence of the same
length with respect to a Noetherian monomial order. To do that we
shall need Robbiano's characterization of monomial orders.

\begin{lemma}[\cite{Robbiano:85}] \label{Robbiano}%
  For every monomial order $\prec$ in $s$ variables, there exists a
  real matrix $M$ having $s$ rows such that $x_1^{m_1} \cdots
  x_s^{m_s}\ \prec \ x_1^{m_1'} \cdots x_s^{m_s'}$ if and only if
  $$(m_1,\ldots,m_s)\cdot M <_{\lex} (m_1',\ldots,m_s')\cdot M,$$
 where $<_{\lex}$
  is the lexicographic order. Moreover, the first column of $M$ is
  nonzero and all its entries are nonnegative.
\end{lemma}

\begin{prop} \label{change}%
  Let $a_1,\ldots,a_s \in R$ be an independent sequence with respect
  to an arbitrary monomial order~$\prec$. Then there exists an index~$i$
  such that the sequence $a_1a_i,\ldots,a_sa_i$ is independent with
  respect to some Noetherian monomial order~$\prec'$.
\end{prop}

\begin{proof}
  By Lemma \ref{Robbiano}, there exists a real vector
  $(v_1,\ldots,v_s)$ having nonnegative components with some $v_i > 0$
  such that $x_1^{m_1} \cdots x_s^{m_s}\prec x_1^{m_1'} \cdots
  x_s^{m_s'}$ implies $\sum_{j =1}^sm_j v_j \le \sum_{j=1}^sm_j' v_j$.
  Define $\prec'$ by the rule:
  \[
  x_1^{m_1} \cdots x_s^{m_s}\ \prec' \ x_1^{m_1'}
  \cdots x_s^{m_s'}\quad \text{if} \quad (x_1x_i)^{m_1}
  \cdots (x_sx_i)^{m_s}\ \prec \ (x_1x_i)^{m_1'} \cdots
  (x_sx_i)^{m_s'}.
  \]
  It is easy to see that $\prec'$ is a monomial order. If
  $x_1^{m_1} \cdots x_s^{m_s}\ \prec' \ x_1^{m_1'}
  \cdots x_s^{m_s'}$, then $\sum_{j=1}^sm_j(v_i + v_j) \le
  \sum_{i=1}^sm_j' (v_i + v_j)$. Since $v_i+v_j > 0$ for all $j
  =1,\ldots,s$, $\prec'$ is Noetherian.

  Let $f$ be a polynomial in $R[X]$ such that $f(a_1a_i,\ldots,a_sa_i)
  = 0$. Put $g = f(x_1x_i,\ldots,x_sx_i)$. Then $\ini_\prec(g)$ has
  the same coefficient as $\ini_{\prec'}(f)$. Since $g(a_1,\ldots,a_s)
  = 0$, the coefficient of $\ini_\prec(g)$ is not invertible. This
  shows that the coefficient of $\ini_{\prec'}(f)$ is not invertible.
\end{proof}

Now we are ready to extend Lombardi's characterization of the Krull
dimension to an arbitrary monomial order.

\begin{theorem} \label{monomial order 1}%
  Let $R$ be a Noetherian ring and~$s$ a positive integer. 
  \begin{enumerate}
  \item \label{Order1A} If $s \le \dim R$, there exists a sequence
    $a_1 \upto a_s \in R$ that is independent with respect to every
    monomial order.
  \item \label{Order1B} If $s > \dim R$, every sequence $a_1 \upto a_s
    \in R$ is dependent with respect to every monomial order.
  \end{enumerate}
\end{theorem}

\begin{proof}
  If $s \le \dim R$, there exists a sequence $a_1 \upto a_s \in R$
  which is weighted independent with respect to every weight sequence
  by \tref{weight1}\eqref{weight1A}. By
  \cref{Noetherian}\eqref{NoetherianA}, this implies that $a_1 \upto
  a_s$ is independent with respect to every monomial order. \par

  If $s > \dim R$, every sequence $a_1 \upto a_s \in R$ is weighted
  dependent with respect to every weight sequence by
  \tref{weight1}\eqref{weight1B}.  If $a_1 \upto a_s$ is independent
  for some monomial order, then $a_1a_i \upto a_sa_i$ is independent
  with respect to some Noetherian monomial order for some $i$ by
  \pref{change}. By \cref{Noetherian}\eqref{NoetherianB}, $a_1a_i
  \upto a_sa_i$ is weighted independent with respect to some weight
  sequence, a contradiction.
\end{proof}

As a consequence, $\dim R$ is the supremum of the length of
independent sequences with respect to an arbitrary monomial order. In
the following we show how this result can be used to prove the
existence of certain relations which look like polynomial identities
in $R$. \par

Let~$\prec$ be an arbitrary monomial order.  For every term~$g$ of $R[X]$ there is a
unique set $\mathcal{M}(g)$ of monomials $h \succ g$ such that
\begin{enumerate}
  \renewcommand{\theenumi}{\roman{enumi}}
\item every monomial $u \succ g$ is divisible by a monomial of $\mathcal{M}(g)$, 
\item the monomials of $\mathcal{M}(g)$ are not divisible by each other. 
\end{enumerate}
\noindent For every polynomial $f \in R[X]$ vanishing at $a_1 \upto a_s$,  
we can always find a polynomial vanishing at $a_1 \upto a_s$ of the form 
$$g + \sum_{h \in \mathcal{M}(g)}c_hh$$
where $g = \ini_\prec(f)$ and $c_h \in R$.  To see this, one only needs
to write every term $u \succ g$  of $f$ in the form $u = vh$ for some
$h \in \mathcal{M}(g)$ and replace $u$ by the term $v(a_1 \upto
a_s)h$.  Therefore, $a_1 \upto a_s$ is a dependent sequence with respect to $\prec\ $ if and only if there exists a polynomial of the above form  vanishing at $a_1 \upto a_s$ such that the coefficient of $g$ is 1. 
Since the monomials of $\mathcal{M}(g)$ can be written down in a canonical way from the exponent vector of $g$, 
this polynomial yields an algebraic relation between elements of $R$ which are similar to a polynomial identity. 

\begin{ex}
  Let $\prec$ be the lexicographic order. For a monomial $g = x_1^{m_1} \cdots x_s^{m_s}$, $\mathcal{M}(g)$ is the set of the monomials 
$x_1^{m_1 + 1}, x_1^{m_1}x_2^{m_2+1} \upto x_1^{m_1} \cdots x_{s-1}^{m_{s-1}} x_s^{m_s +1}.$ 
Therefore, $a_1  \upto a_s$ is a dependent sequence with respect to the lexicographic order  if and only if there exists a relation of the form
$$a_1^{m_1} \cdots a_s^{m_s}+c_1a_1^{m_1 + 1}  + c_2a_1^{m_1}a_2^{m_2+1} +  \cdots + c_sa_1^{m_1} \cdots a_{s-1}^{m_{s-1}} a_s^{m_s +1} =  0,$$
where $c_1 \upto c_s \in R$. This explains why Theorem \ref{monomial order 1} is a generalization of Lombardi's result in \cite{Lombardi:2002}.  In that paper Lombardi calls $a_1 \upto a_s$ a pseudo-regular sequence if   
$$a_1^{m_1}\cdots a_s^{m_s} + c_1a_1^{m_1+1}a_2^{m_2}\cdots a_s^{m_s}+ \cdots + c_sa_1^{m_1}\cdots a_{s-1}^{m_{s-1}}a_s^{m_s+1} \neq 0$$
for all nonnegative integers $m_1 \upto m_s$ and $c_1 \upto c_s \in R$. By the above observation,
$a_1 \upto a_s$ is pseudo-regular if and only if it is independent with respect to the lexicographic order.
\end{ex}

Similarly as for weighted independent sequences, one may ask whether
$\dim R/0:J^\infty$ is the supremum of the length of independent sequences
in an ideal $J \subseteq R$ with respect to an arbitrary monomial order. 
Unlike the case of weighted independent sequences, we could not give a full answer to this question.
This shows again that independence with respect to a monomial order is more subtle than weighted independence.

\begin{prop} \label{monomial order 2}%
Let $J$ be an arbitrary ideal of $R$. The length of independent sequences in $J$ with respect to an arbitrary monomial order 
is bounded above by $\dim R/0:J^\infty$.
\end{prop}

\begin{proof}
Let $a_1,\ldots,a_s$ be an independent sequence in $J$ with respect to an arbitrary monomial order $\prec$.  By Lemma \ref{change}, $a_1a_i,\ldots,a_sa_i$ is an independent sequence with respect to some Noetherian monomial order for some $i$. By Corollary \ref{Noetherian}, $a_1a_i,\ldots,a_sa_i$ is weighted independent for some weight sequence. Since $a_1a_i,\ldots,a_sa_i \in J$,  $s \le \dim R/0:J^\infty$ by Theorem \ref{weight2}. 
\end{proof}

\section{Generalization to monomial preorders} \label{3sPreorder}

In the previous sections we have considered weight sequences and
monomial orders, and shown analogous results in both cases. 
So one may ask whether there is a common generalization of these results.
We shall see that the following notion provides the platform for 
such a generalization.  \par

Recall that a {\em strict weak order} is a binary relation~$\prec$ on
a set $M$ such that for $f,g,h \in M$ with $f \prec g$ we have:
\begin{enumerate}
  \renewcommand{\theenumi}{\roman{enumi}}
\item \label{Preorder1} $f \prec h$ or $h \prec g$, and
\item \label{Preorder2} $g \not\prec f$ (i.e., $g \prec f$ does not hold).
\end{enumerate}
This is equivalent to say that $\prec$ is a strict partial order   
in which the incomparability relation (given by $f \not\prec g$ and $g \not\prec
f$) is an equivalence relation and the equivalence classes of
incomparable elements are totally ordered.
%These laws imply (and are, in fact, equivalent to) that the
%incomparability relation (given by $f \not\prec g$ and $g \not\prec
%f$) is an equivalence relation and the equivalence classes of
% incomparable elements are totally ordered.

We call a strict weak order~$\prec$ on the set of monomials of the
variables $x_1,x_2, \ldots$ a {\em monomial preorder} if it satisfies
the following conditions:
\begin{enumerate}
  \renewcommand{\theenumi}{\roman{enumi}}
  \addtocounter{enumi}{2}
\item \label{Preorder3} $1 \prec f$ for all monomials $f \ne 1$, and
 \item \label{Preorder4} for all monomials $f,g,h$ the equivalence
  \[
  f \prec g \quad \Longleftrightarrow \quad f h \prec g h
  \]
  holds.
\end{enumerate}
Notice that the actual preorder~$\precsim$ is given by $f \precsim g \
\Leftrightarrow g \not\prec f$, not by $f \prec g$. This slight
inaccuracy in terminology follows common practice in Gr\"obner basis
theory. \par

Obviously, every monomial order is a monomial preorder.  A weight
sequence $\w = w_1, w_2,\ldots$ gives rise to a preorder $\prec_\w$ by
comparing their weighted degree, i.e.
     \[
    \prod_i x_i^{m_i} \prec_\w \prod_i x_i^{m_i'} \quad \text{if} \quad
    \sum_i m_iw_i  < \sum_i m_i'w_i .
    \]
We call this the $\w$-weighted preorder. The following example shows that monomial preorders 
are much more general than monomial orders and weighted preorders. 

\begin{ex}
  Let $M$ be a real matrix of $s$ rows such that the first column is nonzero with nonnegative entries and every row is nonzero with the first nonzero entry positive. Then $M$ defines a monomial preorder in a polynomial ring of $s$ variables by
\[
    f  \prec\ g \quad \text{if} \quad \exp(f)  \cdot M  <_{\lex} \exp(g) \cdot M,
    \]
    where $f, g$ are monomials, $\exp(f)$ and
    $\exp(g)$ denote the exponent vectors of $f,g$, and $<_{\lex}$ is the
    lexicographic order. Note that the assumption on the rows of $M$ is equivalent to \eqref{Preorder3}. 
   Then every monomial order  arises in such a way by \lref{Robbiano}. 
If $M$ has only one column and if its entries are positive integers, then we get a weighted
    preorder.
\end{ex}

\begin{lemma} \label{3lRefine}%
  Every monomial preorder $\prec$ can be refined to a monomial
  order~$\prec^*$, i.e. $f \prec g$ implies $f  \prec^* g$.
\end{lemma}

\begin{proof}
  We choose an arbitrary monomial ordering~$\prec'$ and use it to
  break ties in the equivalence classes of incomparable elements. More
  precisely, we define $f \prec^* g$ if $f \prec g$ or if $f,g$ is
  incomparable and $f \prec' g$.  It is straightforward to check
  that~$\prec^*$ is a monomial order and refines~$\prec$.
\end{proof}

A monomial preorder can be approximated by a weighted preorder by the
following lemma, which is well-known in the case of monomial orders
(\lref{approximate}). 
% To prove this lemma is a challenge because the weight sequence $\w$
% have to be chosen such that incomparable monomials with respect
% to~$\prec$ have the same weighted degree.

\begin{lemma} \label{lCone}%
  Let~$\prec$ be a monomial preorder and let $\mathcal{M}$ be a finite
  set of monomials.  Then there exists a weight sequence $\w$ such
  that the restrictions of~$\prec$ and~$\prec_\w$ to $\mathcal{M}$
  coincide.
\end{lemma}

\begin{proof}
  Assume that $\mathcal{M}$ is a set of monomials in $s$ variables
  $x_1 \upto x_s$. 
 We consider the ``positive cone''
  \[
  \mathcal{P} := \left\{\exp(f) - \exp(g) \mid f,g \ \text{are
      monomials such that} \ g \prec f\right\} \subseteq \ZZ^s
  \]
  and the ``nullcone''
  \[
  \mathcal{N} := \left\{\exp(f) - \exp(g) \mid f,g \ \text{are
      incomparable monomials} \right\} \subseteq \ZZ^s.
  \]
  We also consider the sets
  \[
  \mathcal{P}^+ := \left\{\sum_{i=1}^n \alpha_i u_i
    \left|\strut\right. n \in \NN_{>0}, \ u_i \in \mathcal{P}, \
    \alpha_i \in \RR_{>0}\right\} \subseteq \RR^s
  \]
  and
  \[
  \mathcal{N}^* := \left\{\sum_{i=1}^n \alpha_i v_i
    \left|\strut\right. n \in \NN_{>0}, \ v_i \in \mathcal{N}, \
    \alpha_i \in \RR\right\} \subseteq \RR^s.
  \]
  Assume that $\mathcal{P}^+ \cap \mathcal{N}^* \ne \emptyset$. Then
  there exist vectors $u_1 \upto u_n \in \mathcal{P}$ and $v_1 \upto
  v_m \in \mathcal{N}$ and real numbers $\alpha_1 \upto \alpha_n \in
  \RR_{>0}$ and $\beta_1 \upto \beta_m \in \RR$ such that
  \begin{equation} \label{eqAlphabeta}%
    \sum_{i=1}^n \alpha_i u_i - \sum_{j=1}^m \beta_j v_j = 0.
  \end{equation}
  So $(\alpha_1 \upto \alpha_n,\beta_1 \upto \beta_m) \in \RR^{n+m}$
  is a solution of a system of linear equations with coefficients in
  $\ZZ$ that satisfies the additional positivity conditions $\alpha_i
  > 0$. The existence of a solution in $\RR^{n+m}$ satisfying the
  positivity conditions implies that there also exists a solution in
  $\QQ^{n+m}$ satisfying these conditions. So we may assume $\alpha_i
  \in \QQ_{> 0}$ and $\beta_i \in \QQ$, and then, by multiplying by a
  common denominator, $\alpha_i \in \NN_{> 0}$ and $\beta_i \in
  \ZZ$. It follows from the definition of a monomial preorder that
  $\mathcal{P}$ is closed under addition and that $\mathcal{N}$ is
  closed under addition and subtraction. Therefore, $\sum_{i=1}^n
  \alpha_i u_i \in \mathcal{P}$ and $\sum_{j=1}^m \beta_j v_j \in
  \mathcal{N}$. Hence~\eqref{eqAlphabeta} implies $\mathcal{P} \cap
  \mathcal{N} \ne \emptyset$. So there exist monomials $g \prec f$ and
  incomparable monomials $h,k$ such that
  \[
  \exp(f) - \exp(g) = \exp(h) - \exp(k).
  \]
  This implies $f k = g h$. By condition~\eqref{Preorder4} of the
  definition of a monomial preorder, $g h$ and $g k$ are incomparable,
  hence so are $fk, gk$. This implies that $f , g$ are incomparable, a
  contradiction. Thus, we must have $\mathcal{P}^+ \cap \mathcal{N}^*
  = \emptyset$.

  Now we form the finite set
  \[
  \mathcal{T} := \left\{\exp(f) - \exp(g) \mid f,g \in \mathcal{M}, \
    g \prec f\right\} \cup \{e_1 \upto e_s\},
  \]
  where $e_1 \upto e_s \in \RR^s$ are the standard basis vectors. Then
  $\mathcal{T} \subseteq \mathcal{P}$ since $1 \prec x_i$ for
  all~$i$. We write $T = \{u_1 \upto u_n\}$ and form the convex hull
  \[
  \mathcal{H} := \left\{\sum_{i=1}^n \alpha_i u_i
    \left|\strut\right. \alpha_i \in \RR_{\ge 0}, \ \sum_{i=1}^n
    \alpha_i = 1\right\} \subseteq \mathcal{P}^+.
  \]
  Since $\mathcal{H}$ is a compact subset of $\RR^s$ and
  $\mathcal{N}^*$ is a linear subspace, there exist $u \in
  \mathcal{H}$ and $v \in \mathcal{N}^*$ such that the Euclidean
  distance between~$u$ and~$v$ is minimal. \par

  Set $w := u - v$. Then
  \begin{equation} \label{eqPerp}%
    w \in \left(\mathcal{N}^*\right)^\perp
  \end{equation}
  (the orthogonal complement), since otherwise there would be points
  in $\mathcal{N}^*$ that are closer to~$u$ than~$v$. Set $d :=
  \langle w,w \rangle$, where $\langle \cdot,\cdot\rangle$ denotes the
  standard scalar product. From $\mathcal{P}^+ \cap \mathcal{N}^* =
  \emptyset$ we conclude that $d > 0$. Moreover, \eqref{eqPerp}
  implies $\langle w,u\rangle = \langle w,u - v\rangle = d$. Take $u'
  \in \mathcal{H}$ arbitrary. Then for every $\alpha \in \RR$ with $0
  \le \alpha \le 1$ the linear combination $u + \alpha (u' - u)$ also
  lies in $\mathcal{H}$, so
  \begin{align*}
    d \le \langle u + \alpha (u' - u) - v,u + \alpha (u' - u) -
    v\rangle & = \langle w + \alpha (u' - u) , w+ \alpha (u' - u)
    \rangle \\
    & = d + 2 \alpha \left(\langle w,u'\rangle - d\right) + \alpha^2
    \langle u' - u,u' - u\rangle.
  \end{align*}
  Since this holds for arbitrarily small~$\alpha$, we conclude
  $\langle w,u'\rangle \ge d > 0$. In particular,
  \begin{equation} \label{eqPositive}
    \langle w,u_i\rangle > 0 \quad \text{for} \quad i = 1 \upto n.
  \end{equation}
  Since $\mathcal{N}^*$ has a basis in $\ZZ^s$, the existence of a
  vector $w \in \RR^s$ satisfying~\eqref{eqPerp}
  and~\eqref{eqPositive} implies that there also exists such a vector
  in $\QQ^s$, and then even in $\ZZ^s$. So we may assume $w \in \ZZ^s$
  and retain~\eqref{eqPerp} and~\eqref{eqPositive}. Since the standard
  basis vectors $e_j$ occur among the $u_i$, \eqref{eqPositive}
  implies that $w$ has positive components.

  Let $\w = w_1,w_2,\ldots$ be a weight sequence starting with $w_1
  \upto w_s$ chosen above. Let $f,g$ be two arbitrary monomials of
  $\mathcal{M}$. Then $f \prec_\w g$ if and only if $\langle
  w,\exp(f)\rangle < \langle w,\exp(g)\rangle$. If~$f$ and~$g$ are
  incomparable with respect to~$\prec$, then $\exp(f) - \exp(g) \in
  \mathcal{N}^*$, hence $\langle w,\exp(f)\rangle = \langle
  w,\exp(g)\rangle$ by~\eqref{eqPerp}.  This implies that~$f$ and~$g$
  are incomparable with respect to~$\prec_\w$. If $f \prec g$, then
  $\exp(g) - \exp(f) \in \mathcal{T}$, hence $\langle
  w,\exp(g)-\exp(f)\rangle > 0$ by~\eqref{eqPositive}. This implies
  that $f \prec_\w g$. So we can conclude that~$\prec$ and~$\prec_\w$
  coincide on $\mathcal{M}$.
 \end{proof}

\begin{remark*}
  It is clear that any binary relation on the set of monomials
  satisfying the assertion of \lref{lCone} is a monomial
  preorder. Since the lemma is crucial for obtaining the results of
  this section, this shows that we are working in just the right
  generality.
\end{remark*}

Let $R$ be a Noetherian ring and $R[X] := R[x_1 \upto x_s]$.
Let~$\prec$ be a monomial preorder. For a polynomial $f \in R[X]$ we
define $\ini_\prec(f)$ to be the sum of all terms of~$f$ that are
associated with the minimal monomials appearing in~$f$. As in the
previous sections, we call a sequence $a_1 \upto a_s \in R$ {\em
  dependent} with respect to~$\prec$ if there exists a polynomial $f
\in R[X]$ vanishing at $a_1 \upto a_s$ such that $\ini_\prec(f)$ has
at least one invertible coefficient. Otherwise, the sequence is called
{\em independent} with respect to~$\prec$.  These notions cover both
weighted (in-)dependent sequences and (in-)dependent sequences with
respect to a monomial order. \par

The following result allows us to reduce the study of these notions to
weighted independent sequences and dependent sequences with respect to
a monomial order.

\begin{prop} \label{3pEquiv}%
  Let $a_1 \upto a_s \in R$ be a sequence of elements.
  \begin{enumerate}
  \item \label{3pEquivA} The sequence is independent with respect to
    every monomial preorder if it is weighted independent
    with respect to every weight sequence.
  \item \label{3pEquivB} The sequence is dependent with respect to
    every monomial preorder if it is dependent with
    respect to every monomial order.
  \end{enumerate}
\end{prop}

\begin{proof}
  (a) Assume that $a_1 \upto a_s$ is weighted independent with respect
  to every weight sequence. If $a_1 \upto a_s$ is dependent with
  respect to some monomial preorder~$\prec$, there is a polynomial $f
  \in R[X]$ vanishing at $a_1 \upto a_s$ such that $\ini_\prec(f)$ has
  an invertible coefficient. By \lref{lCone} there exists a weight
  sequence~$\w$ such that $\ini_\prec(f) = \ini_\w(f)$. So $a_1 \upto
  a_s$ is weighted dependent with respect to~$\w$, a contradiction.

  (b) Assume that $a_1 \upto a_s$ is dependent with respect to every
  monomial order. If $a_1 \upto a_s$ is independent with respect to
  some monomial preorder~$\prec$, we use \lref{3lRefine} to find a
  monomial order~$\prec^*$ that refines~$\prec$. If $f \in R[X]$ is a
  polynomial vanishing at $a_1 \upto a_s$, then $\ini_\prec(f)$ has no
  invertible coefficient. Since the least term $\ini_{\prec^*}(f)$
  of~$f$ is minimal with respect to~$\prec^*$, is is also minimal with
  respect to~$\prec$, so it is a term of $\ini_\prec(f)$. Therefore
  the coefficient of $\ini_{\prec^*}(f)$ is not invertible. But this
  means that the sequence is independent with respect to~$\prec^*$, a
  contradiction.
\end{proof}

Combining \pref{3pEquiv} with Theorems~\ref{weight1}\eqref{weight1A}
and~\ref{monomial order 1}\eqref{Order1B}, we obtain the following
generalization of the main results of the two previous sections.

\begin{theorem} \label{3tPreorder}%
  Let $R$ be a Noetherian ring and~$s$ a positive integer.
  \begin{enumerate}
  \item \label{3tPreorderA} If $s \le \dim R$, there exists a sequence
    $a_1 \upto a_s \in R$ that is independent with respect to every
    monomial preorder.
  \item \label{3tPreorderB} If $s > \dim R$, every sequence $a_1 \upto a_s
    \in R$ is dependent with respect to every monomial preorder.
  \end{enumerate}
\end{theorem}

As a consequence, $\dim R$ is the supremum of the length of
independent sequences with respect to an arbitrary monomial preorder.

\section{Algebras over a Jacobson ring} \label{4sJacobson}

In this section we extend our investigation to algebras over a ring. 
Our aim is to generalize the characterization of the Krull dimension 
of algebras over a field by means of the transcendence degree.

Let $A$ be an algebra over a ring $R$. Given a monomial
preorder~$\prec$, we say that a sequence $a_1 \upto a_s$ of elements
of $A$ is {\em dependent over} $R$ with respect to~$\prec$ if there
exists a polynomial $f \in R[X] := R[x_1 \upto x_s]$ vanishing at $a_1
\upto a_s$ such that $\ini_\prec(f)$ has at least one coefficient that
is invertible in $R$. Otherwise, the sequence is called {\em
  independent over} $R$ with respect to~$\prec$. Note that if $R$ is a
field, these are just the usual notions of algebraic dependence and
independence, and they do not depend on the choice of the monomial
preorder.  In this case, it is well known that $\dim A$ is equal to
the transcendence degree of $A$ over $R$.  So we may ask whether $\dim
A$ is equal to the maximal length of independent sequences over $R$
with respect to~$\prec$. \par

The following example shows that this question has a negative answer
in general.

\begin{ex}
  Let $R$ be an one-dimensional local domain. Let $A = R[a^{-1}]$,
  where $a \neq 0$ is an
  element in the maximal ideal of $R$. Then $\dim A = 0$, whereas $a$
  is an independent element over $R$ with respect to every monomial
  preorder. (In fact, there exists only one monomial preorder in just
  one variable.)
\end{ex}

We shall see that the above question has a positive answer if $R$ is a
Noetherian Jacobson ring.  Recall that $R$ is called a {\em Jacobson
  ring} (or Hilbert ring) if every prime of $R$ is the intersection of
maximal ideals.  It is well known that every finitely generated
algebra over a field is a Jacobson ring (see Eisenbud~\cite[Theorem
4.19]{eis}).  More examples are given by tensor products of extensions
of a field with finite transcendence degree \cite{Trung:84}. \par

Clearly, $R$ is a Jacobson ring if and only if every nonmaximal prime
$P$ of $R$ is the intersection of primes $P' \supset P$ with
$\height(P'/P) = 1$. Therefore, the following lemma will be useful in
studying Jacobson rings.  This lemma seems to be folklore. Since we
could not find any references in the literature, we provide a proof
for the convenience of the reader.

\begin{lemma} \label{4lNgo}%
  Let $R$ be a Noetherian ring and $P$ a nonmaximal prime of $R$.
  \begin{enumerate}
  \item \label{4lNgoA} For every prime $Q \supset P$ with
    $\height(Q/P) \ge 2$, there exist infinitely many primes $P'$ with
    $P \subseteq P' \subseteq Q$ and $\height(P'/P) = 1$ in $Q$.
  \item \label{4lNgoB} If $\mathcal{M}$ is a set of primes $P' \supset
    P$ with $\height(P'/P) = 1$, then $P = \cap_{P' \in \mathcal{M}}
    P'$ if and only if $\mathcal{M}$ is infinite.
  \end{enumerate}
\end{lemma}

\begin{proof}
  (a) By factoring out $P$ and localizing at $Q$ we may assume that
  $P$ is the zero ideal of a local domain $R$ with maximal ideal $Q$.
  We have to show that the set of height one primes of $R$ is
  infinite.  By Krull's principal theorem, every element $a \neq 0$ in
  $Q$ is contained in some height one prime $P'$.  So $Q$ is contained
  in the union of all height one primes of $R$.  If the number of
  these primes were finite, it would follow by the prime avoidance
  lemma that $Q$ is contained in one of them, contradicting the
  hypothesis $\height(Q) \ge 2$.  \par
           
  (b) Let $I = \cap_{P' \in \mathcal{M}} P'$.  If $P \neq I$, every
  prime $P'$ of $\mathcal{M}$ is a minimal prime over $I$.  Hence
  $\mathcal{M}$ is finite because $R$ is Noetherian.  Conversely, if
  $\mathcal{M}$ is finite, then $\mathcal{M}$ is the set of minimal
  primes over $I$.  This implies $\height(P) < \height(I)$, hence $P
  \neq I$.
\end{proof}

\begin{cor} \label{4cJacobson}%
  A Noetherian ring $R$ is a Jacobson ring if and only if for every
  prime $P$ with $\dim R/P = 1$ there exist infinitely many maximal
  ideals containing $P$.
\end{cor}

\begin{proof}
  By \lref{4lNgo}, every prime $P$ of a Noetherian ring $R$ with $\dim
  R/P \ge 2$ is the intersection of primes $P' \supset P$ with
  $\height(P'/P) = 1$. Therefore, $R$ is a Jacobson ring if and only
  if every prime $P$ with $\dim R/P = 1$ is the intersection of
  maximal primes.  By \lref{4lNgo}\eqref{4lNgoB}, this is equivalent
  to the condition that there exist infinitely many maximal ideals
  containing $P$.
\end{proof}

We use the above results to prove the following lemma which will play
a crucial role in our investigation on independent sequences over $R$.

\begin{lemma} \label{4lBoundary}
  Let~$a$ be an element of a Noetherian ring $R$ and set
  \[
  U_a := \left\{a^n (1 + a x) \mid n \in \NN_0, \ x \in R\right\}.
  \]
  Then the localization $U_a^{-1} R$ is a Jacobson ring.
\end{lemma}

\begin{proof}
  We will use the inclusion-preserving bijection between the primes of
  $S := U_a^{-1}R$ and the primes $P$ of $R$ satisfying $U_a \cap P =
  \emptyset$.  Let $P$ be such a prime of $R$ with
  $\dim\left(S/U_a^{-1} P\right) = 1$. Then there exists a prime $P_1
  \supset P$ of $R$ with $\height(P_1/P) = 1$ and $U_a \cap P_1 =
  \emptyset$.  The latter condition implies $a \notin P_1$ and $1
  \notin (P_1,a)$. Let $Q$ be a prime of $R$ containing $(P_1,a)$.
  Then $\height(Q/P) \ge 2$.  By \lref{4lNgo}\eqref{4lNgoA}, the set
  \[
  \mathcal{M} := \left\{P' \in \Spec(R) \mid P \subset P' \subset Q, \
    \height(P'/P) = 1\right\}
  \]
  is infinite. Consider the set $\mathcal{N} := \{P' \in \mathcal{M}
  \mid U_a \cap P' \neq \emptyset\}$.  If $\mathcal{N}$ is infinite,
  $P = \bigcap_{P' \in \mathcal{N}} P'$ by \lref{4lNgo}\eqref{4lNgoB}.
  Since $U_a \cap P = \emptyset$, $a \not\in P$. Therefore, there
  exists a prime $P' \in \mathcal{N}$ such that $a \not\in P'$.  Since
  $U_a \cap P' \neq \emptyset$, this implies $1 + ax \in P'$ for some
  $x \in R$. Hence $1 \in (P',a) \subseteq Q$, a contradiction.  So
  $\mathcal{N}$ must be finite, and we can conclude that $\mathcal{M}
  \setminus \mathcal{N}$ is infinite.  By the definition of
  $\mathcal{M}$ and $\mathcal{N}$, the set of primes $P' \supset P$
  with $\height(P'/P) = 1$ and $U_a \cap P' = \emptyset$ is infinite.
  Since this set corresponds to the set of maximal ideals of $S$
  containing $U_a^{-1} P$, the assertion follows from
  \cref{4cJacobson}.
\end{proof}

\begin{remark*}
  The localization $U_a^{-1} R$ from \lref{4lBoundary} was already
  used by Coquand and Lombardi to give a short proof for the fact that
  the Krull dimension of a polynomial ring over a field is equal to
  the number of variables~\cite{Coquand-Lombardi:05}.  They called it
  the boundary of~$a$ in $R$.
\end{remark*}

Now we are going to give a characterization of the Krull dimension of
algebras over a Jacobson ring $R$ by means of independent elements
over $R$ with respect to an arbitrary monomial preorder~$\prec$. First
we need to consider the case where~$\prec$ is the lexicographic order
with $x_i > x_{i+1}$ for all~$i$.
 
We call an $R$-algebra {\em subfinite} if it is a subalgebra of a
finitely generated $R$-algebra. 
A subfinite algebra needs not to be finitely generated.

\begin{theorem} \label{4tLex}%
  Let $A$ be a subfinite algebra over a Noetherian Jacobson ring $R$
  and let~$s$ be a positive integer. There exists a sequence $a_1
  \upto a_s \in A$ that is independent over $R$ with respect to the
  lexicographic order if and only if $s \le \dim A$.
\end{theorem}

\begin{proof}
  If $s \le \dim A$, Lombardi~\cite{Lombardi:2002} (which does require $A$ to be Noetherian) tells us
  that there exists a sequence of length~$s$ that is independent over
  $A$ with respect to the lexicographic order. Therefore it is also
  independent over $R$.
  
  The next step is to prove the converse under the hypothesis that $A$
  is finitely generated. We use induction on~$s$. We may assume that
  $A \ne \{0\}$, $\dim A < \infty$, and $s = \dim A + 1$. We have to
  show that every sequence $a_1 \upto a_s \in A$ is dependent over $R$
  with respect to the lexicographic order. \par
 
  Let $T$ be the set of univariate polynomials~$f \in R[x]$ whose
  initial term $\ini(f)$ has coefficient~$1$. Since $T$ is
  multiplicatively closed, so is the set
  \[
  U := \left\{f(a_s) \mid f \in T\right\} \subseteq A.
  \]
  Let $A' := U^{-1} A$. If $\dim A' = s-1$, then $A$ has a height
  $s-1$ prime $P$ with $U \cap P = \emptyset$. This prime must be
  maximal because $\dim A = s - 1$.  Since $R$ is a Jacobson ring,
  $A/P$ is a finite field extension of $R/(R \cap P)$
  \cite[Theorem~4.19]{eis}.  Since $U \cap P = \emptyset$, $a_s
  \not\in P$. These facts imply that there exists $g \in R[x]$ such
  that $a_s g(a_s) - 1 \in P$.  But $1 - x g \in T$, so $1 - a_s
  g(a_s) \in U \cap P$, a contradiction. So we can conclude that $\dim
  A' < s-1$. \par

  If $A' = \{0\}$ (which must happen if $s = 1$), then $0 \in U$,
  hence there exists $f \in T$ with $f(a_s) = 0$.  So the sequence
  $a_1 \upto a_s$ is dependent over $R$ with respect to the
  lexicographic order. Having dealt with this case, we may assume $A'
  \ne \{0\}$.  Let $R' := U^{-1} R[a_s]$. Then $A'$ is finitely
  generated as an $R'$-algebra. By \lref{4lBoundary}, $R'$ is a
  Jacobson ring. So we may apply the induction hypothesis to $A'$.
  This tells us that the sequence $a_1 \upto a_{s-1}$ (as elements of
  $A'$) is dependent over $R'$ with respect to the lexicographic
  order. Thus, there exists a polynomial $g \in R'[x_1 \upto x_{s-1}]$
  vanishing at $a_1 \upto a_{s-1}$ such that the coefficient of
  $\ini_{\lex}(g)$ is invertible in $R'$. We may assume that this
  coefficient is~$1$.  By the definition of $A'$ there exists $c_0 \in
  R$ such that $c_0 g \in R[a_s][x_1 \upto x_{s-1}]$ and $(c_0 g)(a_1
  \upto a_{s-1}) = 0$ (as an element of $A$). Replacing every
  coefficient $c \in R[a_s]$ of the polynomial $c_0g$ by a polynomial
  $c^* \in R[x_s]$ with $c^*(a_s) = c$, we obtain a polynomial $g^*
  \in R[x_1 \upto x_s]$ vanishing at $a_1 \upto a_s$.  Since $c_0 \in
  U$, we may choose $c_0^* \in T$. Clearly, the coefficient of
  $\ini_{\lex}(g^*)$ is equal to the coefficient of $\ini(c_0^*)$,
  which is~$1$. This shows that $a_1 \upto a_s$ are dependent over $R$
  with respect to the lexicographic order. \par

  Now we deal with the case $A$ is a subalgebra of a finitely
  generated $R$-algebra $B$.  Let $P_1 \upto P_n \in \Spec(B)$ be the
  minimal primes of $B$, and assume that we can show that for
  every~$i$, the images of $a_1 \upto a_s$ in $A/(A \cap P_i)$ are
  dependent over $R$ with respect to the lexicographic order. Then for
  every~$i$, there exists a polynomial $f_i \in R[x_1 \upto x_s]$ with
  $f_i(a_1 \upto a_s) \in P_i$ such that the coefficient of
  $\ini_{\lex}(f_i)$ is invertible. Since $\prod_{i=1}^n f_i(a_1 \upto
  a_s)$ lies in the nilradical of $B$, there exists~$k$ such the
  polynomial $f := \prod_{i=1}^n f_i^k$ vanishes at $a_1 \upto a_s$.
  Since the coefficient of $\ini_{\lex}(f)$ is also invertible, this
  shows that $a_1 \upto a_s$ are dependent over $R$ with respect to
  the lexicographic order. So we may assume that $B$ is an integral
  domain.  By Giral~\cite[Proposition~2.1(b)]{Giral:81}
  (or~\cite[Exercise~10.3]{Kemper.Comalg}), there exists a nonzero $a
  \in A$ such that $A[a^{-1}]$ is a finitely generated
  $R$-algebra. Since $\dim A[a^{-1}] \le \dim A < s$, the sequence
  $a_1 \upto a_s$ is dependent over $R$ with respect to the
  lexicographic order. This completes the proof.
\end{proof}

One can use Theorem \ref{4tLex} to prove the existence of nontrivial
relations between algebraic numbers (i.e., elements of an algebraic
closure of $\QQ$).

\begin{ex} \label{exNumber}%
  Let~$a$ and~$b$ be two nonzero algebraic numbers. There exists $d
  \in \ZZ \setminus \{0\}$ such that~$a$ and~$b$ are integral over
  $\ZZ[d^{-1}]$. So $A := \ZZ\left[a,b,d^{-1}\right]$ has Krull
  dimension~$1$. By Theorem \ref{4tLex}, there is a polynomial $f \in
  \ZZ[x_1,x_2]$ vanishing at $a,b$ such that the coefficient of
  $\ini_{\lex}(f)$ is~$1$.  Let $\ini_{\lex}(f) = x_1^m x_2^n$. Then
  all monomials of~$f$ are divisible by $x_1^m$. Hence we may assume
  $m = 0$. Thus,
  \[
  b^n = a \cdot g(a,b) + b^{n+1} \cdot h(a,b)
  \]
  for some $g,h \in \ZZ[x_1,x_2]$. It is not clear how the existence
  of such a relation follows directly from properties of algebraic
  numbers. In the case that $\ZZ[a,b]$ is a Dedekind ring, we derived
  such a relation in \exref{2exZ}.
\end{ex}

To the best of our knowledge, the following immediate consequence of
\tref{4tLex} is new even for finitely generated algebras.

\begin{cor} \label{4cAB}%
  Let $R$ be a Noetherian Jacobson ring. If $A \subset B$ are
  subfinite $R$-algebras, then
  \[
  \dim A \le \dim B.
  \]
\end{cor}

Now we turn to arbitrary monomial preorders and prove the main result
of this section. The proof relies on \cref{4cAB}.

\begin{theorem} \label{4tMain}%
  Let $A$ be a subfinite algebra over a Noetherian Jacobson ring $R$
  and let~$s$ be a positive integer.
  \begin{enumerate}
  \item \label{4tMainA} If $s \le \dim A$, there
    exists a sequence $a_1 \upto a_s \in A$ that is independent over
    $R$ with respect to every monomial preorder.
  \item \label{4tMainB} If $s > \dim A$, every sequence $a_1 \upto
    a_s \in A$ is dependent over $R$ with respect to every monomial
    preorder.
  \end{enumerate}
\end{theorem}

\begin{proof}
  (a) Let $A$ be a subalgebra of a finitely generated $R$-algebra $B$.
  Since $\sqrt{0_A} = \sqrt{0_B} \cap A$ and since the nilradical is the intersection of all minimal primes,
  the set of the minimal primes of $A$ is contained in the set of prime ideals of the form $P\cap A$, 
where $P$ is a minimal prime  ideal of $B$.
 Therefore, there exists a minimal prime ideal $P$ of $B$ such that $\dim A/P \cap A = \dim A$.
 If $a_1,...,a_s \in A$ is independent over $R$ in $A/P \cap A$, then it is also independent over $R$ in $A$.
 Therefore, we may replace $A$ by $A/P \cap A$ and assume that $B$ is an integral domain.
 By Giral~\cite[Proposition~2.1(b)]{Giral:81}
  (or~\cite[Exercise~10.3]{Kemper.Comalg}), there exists a nonzero $a
  \in A$ such that $A[a^{-1}]$ is a finitely generated   $R$-algebra. 
Then  $\dim A[a^{-1}] \le \dim A$. By \cref{4cAB}, $\dim A \le \dim A[a^{-1}]$.
Hence $s \le \dim A = \dim A[a^{-1}]$. Choose $a_1,...,a_s \in A$ such that
$(a_1,...,a_s)A[a^{-1}]$ has a minimal prime ideal of height $s$. 
Then $a_1,...,a_s$ is independent over $R$ with respect to every monomial preorder by 
 \pref{parameter} and  \tref{3pEquiv}\eqref{3pEquivA}. \par
(b) Let $A' := R[a_1 \upto a_s] \subseteq A$.  By
  \cref{4cAB}, $\dim A' \le \dim A$.  So
  \tref{3tPreorder}\eqref{3tPreorderB} yields for every monomial
  preorder~$\prec$ a polynomial $f \in A'[x_1 \upto x_s]$ vanishing at
  $a_1 \upto a_s$ such that $\ini_\prec(f)$ has an invertible
  coefficient $c_0$. We may assume $c_0 = 1$. Write $f = \sum_{i=0}^n
  c_i t_i$ with $c_i \in A'$ and $t_i$ pairwise different monomials
  such that $t_0$ is minimal among the $t_i$. Choose polynomials
  $c^*_i \in R[x_1 \upto x_s]$ with $c_i^*(a_1 \upto a_s) = c_i$ and
  $c_0^* = 1$. Set set $f^* = \sum_{i=0}^n c_i^* t_i$. Then $f^*$ is a
  polynomial of $R[x_1 \upto x_s]$ vanishing at $a_1 \upto a_s$. From
  the compatibility of monomial preorders with multiplication we
  conclude that $t_0$ is a term of $\ini_\prec(f^*)$. This shows that
  the sequence $a_1 \upto a_s$ is dependent over $R$ with respect
  to~$\prec$.
\end{proof}

\tref{4tMain} generalizes a result of Giral~\cite{Giral:81} which says
that the dimension of a subfinite algebra over a field is equal to its
transcendence degree. \par

As a consequence of the above results, we give a characterization of Jacobson rings,
 which implies that the hypothesis that $R$ is a
Jacobson ring cannot be dropped from \cref{4cAB} and \tref{4tMain}.

\begin{cor} \label{3pConverse} For a Noetherian ring $R$, the
  following statements are equivalent:
  \begin{enumerate}
  \item \label{3pConverseA} $R$ is a Jacobson ring.
  \item \label{3pConverseB} For every subfinite $R$-algebra $A$ and
    every monomial preorder, $\dim A$ is the supremum of the length of independent
    sequences over $R$ in $A$.
  \item \label{3pConverseC} If $A \subseteq B$ is a pair of subfinite
    $R$-algebras, then $\dim A \le \dim B$.
  \end{enumerate}
\end{cor}

\begin{proof}
  The only implication that requires a proof is
  that~\eqref{3pConverseC} implies~\eqref{3pConverseA}. But if $R$ is
  not a Jacobson ring, then by \cite[Lemma~4.20]{eis}, $R$ has a
  nonmaximal prime ideal $P$ such that $A := R/P$ contains a nonzero
  element~$b$ for which $B := A[b^{-1}]$ is a field. So
  \eqref{3pConverseC} fails to hold.
\end{proof}


\begin{thebibliography}{2}


\bibitem{Bjork:82}
Bj\"ork, J.-E.: \emph{On the maximal number of $\frak a$-independent elements in ideals of noetherian ring}.
In: Seminaire d’Algebre Paul Dubreil et Marie-Paule Malliavin (Paris, 1981), Lect. Notes Math. \textbf{924}, pp. 413--422, Springer-Verlag, New York/Berlin (1982)

\bibitem{Bosma:1997} Bosma, W., Cannon, J.J. and Playoust, C.: \emph{The Magma Algebra System I: The User
Language}, J. Symb. Comput. \textbf{24} (1997), 235--265.

\bibitem{B-H} Bruns, W. and Herzog, J.: \emph{Cohen-Macaulay rings}, Cambridge University Press, Cambridge (1993).

\bibitem{Coquand-Lombardi:03}
Coquand, T., Lombardi, H.: \emph{Hidden constructions in abstract algebra (3) Krull dimension of distributive lattices and commutative rings}. In: Commutative ring theory and applications (Fez, 2001), Lecture Notes in Pure and Appl. Math. \textbf{231}, pp. 477--499, Dekker, New York (2003)

\bibitem{Coquand-Lombardi:05}
Coquand, T., Lombardi, H.: \emph{A Short Proof for the {K}rull Dimension
  of a Polynomial Ring}.  Amer. Math. Monthly \textbf{112}, 826--829 (2005)

\bibitem{Coquand-Lombardi-Roy}
Coquand, T., Lombardi, H., Roy, M.-F.: \emph{An elementary characterization of Krull dimension}. 
In: From sets and types to topology and analysis, Oxford Logic Guides \textbf{48}, pp. 239--244, Oxford Univ. Press, Oxford (2005)

\bibitem{Coquand-Ducos-Lombardi-Quitte}
Coquand, T., Ducos, L., Lombardi, H., Quitt\'e, C.: \emph{Constructive Krull dimension. I. Integral extensions}, J. Algebra Appl. \textbf{8} , no. 1, 129--138 (2009)
 
\bibitem{Ducos:09}
Ducos, L.: \emph{Sur la dimension de Krull des anneaux noeth\'eriens}, J. Algebra \textbf{322}, no. 4, 1104--1128 (2009)
 
\bibitem{eis}
Eisenbud, D.: \emph{Commutative Algebra with a View Toward Algebraic
  Geometry}, Springer-Verlag, New York (1995)

\bibitem{Giral:81}
Giral, J.~M.: \emph{Krull Dimension, Transcendence Degree and Subalgebras
  of Finitely Generated Algebras}, Arch. Math. (Basel) \textbf{36},
  305--312 (1981)

\bibitem{Kemper.Comalg}
Kemper, G.: \emph{A Course in Commutative Algebra}, Graduate Texts in
  Mathematics~\textbf{256}, Springer-Verlag, Berlin, Heidelberg (2011)

\bibitem{Kemper:2011} 
Kemper,  G.: {\em The Transcendence Degree
    over a Ring}, preprint, Technische {U}niversit{\"a}t {M}{\"u}nchen
  (2011), {\tt http://arxiv.org/abs/1109.1391v1}.

\bibitem{Lombardi:2002}
Lombardi, H.: \emph{Dimension de {K}rull, {N}ullstellens{\"a}tze et
  {\'e}valuation dynamique}, Math. Z. \textbf{242}, 23--46 (2002)

\bibitem{Matsumura:86}
Matsumura, H.: \emph{Commutative Ring Theory}, Cambridge Studies in
  Advanced Mathematics~\textbf{8}, Cambridge University Press, Cambridge (1986)

\bibitem{Robbiano:85}
Robbiano, L.: \emph{Term orderings on the polynomial ring}. In:
  EUROCAL '85, Vol.\ 2 (Linz, 1985),  Lecture Notes in Comput. Sci. \textbf{204}, pp. 513--517, Springer, Berlin
  (1985)
  
  \bibitem{Trung:84}
Trung, N.~V.: \emph{On the tensor product of extensions of a field}.
Quarterly J.  Math. \textbf{35}, 337--339  (1984)

\bibitem{Trung:85}
Trung, N.~V.: \emph{Maximum number of independent elements and dimension of prime divisors in completions of local rings}.
J. Algebra \textbf{93}, 418--438  (1985)

\bibitem{Valla:70} Valla, G.: \emph{Elementi independenti rispetto ad un ideale}.
Rend. Sem. Mat. Univ. Padova \textbf{44}, 339--354 (1970)

\end{thebibliography}
\end{document}